\documentclass[leqno,12pt]{article}

\usepackage{epsfig}

\usepackage{amsmath}
\usepackage{amscd}
\usepackage{amsopn}
\usepackage{amsthm}
\usepackage{amsfonts,amssymb}\usepackage{amsfonts,bbm}
\usepackage{latexsym}

\usepackage{color}

\makeatletter
\@addtoreset{equation}{section}
\makeatother

\newtheorem{lemma}{LEMMA}[section]
\newtheorem{proposition}[lemma]{PROPOSITION}
\newtheorem{corollary}[lemma]{COROLLARY}
\newtheorem{theorem}[lemma]{THEOREM}
\newtheorem{remark}[lemma]{REMARK}
\newtheorem{remarks}[lemma]{REMARKS}

\newtheorem{definition}[lemma]{DEFINITION}
\newtheorem{assumption}[lemma]{ASSUMPTION}
\newcommand{\real}{\mathbbm{R}}

\newcommand{\nat}{\mathbbm{N}}

\renewcommand{\a}{\alpha}
\renewcommand{\b}{\beta}

\newcommand{\g}{\gamma}

\newcommand{\vp}{\varphi}
\newcommand{\ve}{\varepsilon}

\newcommand{\reald}{{\real^d}}

\newcommand{\on}{\quad\text{ on }}
\newcommand{\und}{\quad\mbox{ and }\quad}
\newcommand{\inv}{^{-1}}
\newcommand{\ov}{\overline}

\newcommand{\W}{\mathcal W}  
\newcommand{\C}{\mathcal C}

\renewcommand{\H}{{\mathcal H}}
\newcommand{\B}{\mathcal B}

\newcommand{\M}{\mathcal M}

\newcommand{\K}{\mathcal K}

\newcommand{\dist}{\mbox{\rm dist}}

\newcommand{\supp}{\operatorname*{supp}}

\newcommand{\kapi}{\operatorname*{{cap_\ast}}}
\newcommand{\kapo}{\operatorname*{{cap^\ast}}}

\newcommand{\itemframe}%
{\setlength{\parskip}{10pt}\begin{enumerate} \setlength{\topsep}{10pt}%
\setlength{\itemsep}{15pt}\setlength{\parsep}{5pt}}

\newcommand{\vx}{\ve_x}

\newcommand{\Px}{\mathcal P}

\newcommand{\vc}{{V^c}}

\newcommand{\kap}{\operatorname*{cap}}

\newcommand{\es}{E_{\mathbbm P}}

\title{Liouville property, Wiener's test\\ and unavoidable sets  for Hunt processes} 
\author{Wolfhard Hansen}

\date{}

\begin{document}

\maketitle

\begin{abstract}

Let $(X,\W)$ be a balayage space, $1\in \W$, or -- equivalently -- let $\W$      
be the set of excessive functions of a Hunt process on a locally compact space~$X$ with countable base such that
$\W$  separates points, every function in $\W$ is the supremum of its
continuous minorants and there exist  strictly positive continuous       $u,v\in \W$ 
such that $u/v\to 0$  at infinity. We suppose that there is a~Green function 
$G>0$ for $X$, a metric $\rho$ for $X$ and a decreasing function $g\colon[0,\infty)\to (0,\infty]$ 
having the doubling property such that $G\approx g\!\circ\!\rho$.

Assuming that the constant function $1$ is harmonic and  balls are relatively compact, 
is is shown that every positive harmonic function is constant (Liouville property) and 
that  Wiener's test at infinity shows, if a given
set $A$ in $X$ is unavoidable, that is, if the process hits $A$ with probability one, wherever it starts. 

An application yields that locally finite unions of pairwise disjoint balls $B(z,r_z)$, $z\in Z$, 
which have a certain separation property with respect to a~suitable measure $\lambda$ on $X$
are unavoidable if and only if, for some/any point~$x_0\in X$,  the series $\sum_{z\in Z} g(\rho(x_0,z))/g(r_z) $ diverges.

The results generalize and, exploiting a zero-one law for hitting probabilities,  simplify recent work by
S.\,Gardiner and M.\,Ghergu,  A.\,Mimica and  Z.\,Vondra\v cek, and the author.

 {

  }
\end{abstract}

\section{Preliminaries and main results}

% \begin{center}
%           Problem in Prop,. 2.1:  $\rho$ metric for (the topology????) on $X$ ???
% \end{center} 

Let $X$ be a locally compact space with countable base.
Let $\C(X)$ denote the set of all continuous real functions on $X$
and let $\B(X)$ be the set of all Borel measurable numerical functions on $X$.
The set of all (positive) Radon measures on $X$ will be denoted by  $\M(X)$. 

Moreover, let $\W$ be a convex cone of positive lower 
semicontinuous numerical functions on~$X$ such that $1\in \W$ and $(X,\W)$ is a balayage space
(see \cite{BH}, \cite{H-course}  or \cite[Appendix]{HN-unavoidable}). In particular, the following holds:
\begin{itemize} 
\item[\rm (C)]
 $\W$ separates the points in $X$,
\[ 
              w=\sup\{v\in\W\cap \C(X)\colon v\le w\} \qquad\mbox{ for every } w\in \W,
\]
and there are strictly positive $u,v\in\W\cap \C(X)$ such that $u/v\to 0$  at~infinity.   
\end{itemize} 
Then there exists a Hunt process $\mathfrak X$ 
on $X$ such that $\W$ is the set $\es$ of excessive functions for the transition semigroup 
$\mathbbm P=(P_t)_{t>0}$ of $\mathfrak X$ (see \cite[IV.7.6]{BH} or \cite[Appendix]{HN-unavoidable}), that is,  
\[
             \W=\{v\in \B^+(X)\colon \sup\nolimits_{t>0} P_tv=v\}.
\]

We note that, conversely, given any sub-Markov semigroup $\mathbbm P=(P_t)_{t>0}$  on $X$
such that (C) is satisfied by its convex cone $\es$ of excessive functions, $(X,\es)$ is a~balayage space, 
and $\mathbbm P$ is the transition semigroup of a Hunt process
 (see \cite[Corollary 2.3.8]{H-course}
or \cite[Corollary A.5]{HN-unavoidable}).                   

For every numerical function $f$ on $X$, let
\[
         R_f:=\inf\{v\in \W\colon v\ge f\}.
\]
In particular, for every subset $A$ of $X$, we have reduced functions  $R_u^A$, $u\in \W$, and reduced measures
$\vx^A$, $x\in X$, defined by
\[
          R_u^A:=R_{1_Au}=\inf\{ v\in \W\colon v\ge u\mbox{ on } A\} \und  \int u\,d\vx^A=R_u^A(x).
\]
Clearly,  $R_u^A\le u$ on $X$ and $R_u^A=u$ on
$A$. 
If  $A$ is  open, then 
 \begin{equation}\label{red-W}
R_u^A\in\W.
\end{equation} 
 For a~general subset~$A$,
the greatest lower semicontinuous minorant~$\hat R_1^A$  of $R_1^A$
is contained 
in~$\W$, and  $\hat R_1^A=R_1^A$ on~$A^c$ (see \cite[p.\ 243]{BH}). 

If $A$ is  Borel measurable, then, for every $x\in X$,
\begin{equation}\label{connection}
               R_1^A(x)= P^x[T_A<\infty], 
\end{equation} 
where  $T_A(\omega):=\inf \{t\ge 0\colon X_t(\omega)\in A\}$ and, more generally,
\[
    \vx^A(B)=P^x[X_{T_A}\in B; T_A<\infty]
\]
for every Borel measurable set $B$ in  $X$  (see \cite[VI.3.14]{BH}).

For every open set $U$ in $X$, let $\H^+(U)$ denote the set of all 
functions $h\in\B^+(X)$ which are \emph{harmonic on~$U$} (in the sense of \cite{BH}), that is,
such that $h|_U\in \C(U) $  and
\begin{equation}\label{mv}
    \vx^\vc(h):=\int h\,d\vx^\vc=h(x) 
\end{equation} 
for every open $V$ such that $x\in V$ and $\ov V$ is  a compact in $U$.
 Let $\tilde\H^+(U)$ denote the (possibly larger) set of all $h\in\B^+(X)$ such that (\ref{mv})
holds, whenever $V$ is open,  $x\in V$ and $\ov V$ is compact in $U$.  
By \cite[VI.2.6]{BH}), for every set $A$ in $X$, 
\begin{equation}\label{A-harmonic}
     R_u^A\in \H^+(X\setminus \ov A), \qquad \mbox{ if }  u\in \W, \  u\le w\in \W\cap \C(X).
\end{equation}

{\it In the following let us assume 
that the constant function $1$ is harmonic on $X$.}

We recall that a subset $A$ of $X$ is called \emph{unavoidable}, if $R_1^A=1$ (or, equivalently, $\hat R_1^A=1$). 
Otherwise, it is called \emph{avoidable}, that is, $A$~is avoidable, if there exists $x\in X$
such that $R_1^A(x)<1$. The following zero-one law will play an important role
(for its proof and the proof of the subsequent corollary see \cite[Proposition 2.3]{HN-unavoidable}).

\begin{proposition}\label{so-simple}
For every $A\subset X$, 
\begin{equation}\label{0-1}
                  R_1^A=1 \quad \mbox{ or }\quad  \inf\nolimits_{x\in X} R_1^A(x)=0.
\end{equation} 
\end{proposition} 

\begin{corollary}\label{unav-corollary}
If $A \subset X$ is  unavoidable and $B\subset X$,  $\g>0$ such that   $ R_1^B\ge \g$ on $A$, then $B$ is unavoidable.
\end{corollary} 

Moreover, we recall the following elementary fact (see \cite[Lemma 2.2]{HN-unavoidable}).
\begin{proposition}\label{simple} Let $A$ be an unavoidable set in $X$ and let $(B_n)$ be a sequence 
of relatively compact sets in $X$.\footnote{It is easily seen that it is sufficient to assume that the the functions 
$R_1^{B_n}$, $n\in\nat$, are $\Px$-bounded.}  Then the following hold.
\begin{itemize}
\item[\rm(a)] For every $n\in\nat$, the set  $A\setminus (B_1\cup B_2\cup\dots\cup B_n) $ is unavoidable.
\item[\rm(b)] If $A\subset \bigcup_{n\in\nat} B_n$, then $\sum_{n\in\nat}   R_1^{B_n}=\infty$. 
\end{itemize} 
\end{proposition} 

By definition, a \emph{potential on $X$} is a function $p\in \W$ such that, for every relatively compact open 
set $U$ in $X$, the function $R_p^{X\setminus U} $ is continuous and real on~$U$ 
and 
\[
                       \inf \{ R_p^{X\setminus U}\colon U\mbox{ relatively compact open in }X\}=0.
\]
By \cite[Proposition 4.2.10]{H-course}, a~function $p\in \W\cap \C(X)$ is a potential if and only if
 there exists a strictly positive $q\in \W\cap \C(X)$
such that $p/q$ vanishes at infinity.     Let~$\Px$ denote the set of all continuous    
real potentials on $X$.

In the following let us assume that there is a Green function $G$ for $(X,\W)$
which is related to a metric for the topology of $X$:

\begin{assumption}\label{ass-1}
We have a Borel measurable function $G\colon X\times X\to (0,\infty]$ and  a~metric $\rho$
for $X$ such that the following hold:
\begin{itemize} 
\item      [\rm (i)] 
For every $y\in X$,  $ G(\cdot,y)$ is a potential which is harmonic on $X\setminus  \{y\}$.   
\item [\rm (ii)] 
For every potential $p$ on $X$, there exists a measure $\mu$ on $X$    such that 
\begin{equation}\label{p-rep}
p=G\mu:=\int G(\cdot,y)\,d\mu(y).
\end{equation} 
\item[\rm (iii)] There exist a decreasing function  $g\colon [0,\infty)\to (0,\infty]$ and a constant $c\ge 1$ such that 
\[
c\inv g\!\circ\!\rho\, \le\, G\,\le \, c\, g\!\circ\! \rho.
\]
\end{itemize} 
\end{assumption}

\begin{remarks}{\rm
1. Having (i), each of the following properties implies (ii).
\begin{itemize}
\item
$G$ is lower semicontinuous on $X\times X$, continuous outside the diagonal,
the potential kernel $V_0:=\int_0^\infty P_t\,dt$ of $\mathfrak X$ is 
proper, and there is a measure~$\mu$ on $X$ such that $V_0f=\int G(\cdot, y)f(y)\,d\mu(y)$, $f\in\B^+(X)$
(see \cite{maagli-87}  and \cite[III.6.6]{BH}).
\item
$G$ is locally bounded off the diagonal, each function $G(x,\cdot)$ is lower semicontinuous 
on $X$ and continuous on $X\setminus \{x\}$, and there exists a measure $\nu$ on $X$ such that $G\nu\in \C(X)$ 
and $\nu(U)>0$, for every  finely open  $U\ne \emptyset$ (the latter holds, for example, if 
$V_0(x,\cdot)\ll \nu$, $x\in X$). See \cite[Theorem 4.1]{HN-rep-potential}.
\end{itemize} 

2. For a discussion of (iii) and the later doubling property (\ref{doubling})  see the Appendix.}
\end{remarks}

The measure in (\ref{p-rep}) is uniquely determined and, given any measure~$\mu$
on~$X$ such that $p:=G\mu$ is a potential, the complement of the support of $\mu$ is the largest open
set, where $p$ is harmonic (see, for example, \cite[Proposition 5.2 and Lemma~2.1]{HN-rep-potential}).

Suppose that $A$ is a subset of $X$ such that $\hat R_1^A$ is a potential. Then there is a~unique measure~$\mu_A$ on $X$, 
the \emph{equilibrium measure for $A$}, such that 
\[
               \hat R_1^A=G\mu_A.
\]                             
If $A$ is open, then $\hat R_1^A=R_1^A\in \H^+(X\setminus \ov A)$, and hence   
$\mu_A$ is supported by  $\ov A$. We~observe that, for a general balayage space, this may already fail
if $A$ is compact (see~\cite[V.9.1]{BH}).

We  define inner capacities for open sets $U$ in $X$ by 
\begin{equation}\label{inner-cap}
 \kapi U:=\sup\bigl\{\|\mu\|\colon  \mu\in\M(X), \ \mu(X\setminus U)=0,\ G\mu\le 1\bigr\}
\end{equation} 
and outer capacities for arbitrary sets $A$ in $X$ by
\begin{equation}\label{outer-cap}
 \kapo A:=\inf\bigl\{ \kapi U\colon U\mbox{ open neighborhood of } A\bigr\}.
\end{equation} 
Of course, the function  $U\mapsto \kapi U$, $U$ open in $X$, is increasing. Hence  the function  $A\mapsto \kapo A$,
$A\subset X$, is also increasing  and $\kapo A=\kapi A$, if $A$ is
open. If $\kapi A=\kapo A$, we may simply write $\kap A$ and speak of the capacity of $A$.

The capacity of open sets $U$ is essentially determined by the total mass of  
equilibrium measures for open sets which are relatively compact in $U$:

\begin{lemma}\label{cap-potential}
For every open set $U$ in $X$,
\[
\kap U \ge \sup\{\|\mu_V\|\colon V\mbox{ open and }\ov V \mbox{ compact in } U\}\ge 
c^{-2} \kap U.
\] 
\end{lemma}

\begin{proof} The first inequality is trivial.     To prove the second inequality, 
let  $\mu\in \M(X)$ such that $\mu(X\setminus U)=0$ and $G\mu\le 1$, 
and let $K$ be a compact in $U$.  We may choose an open neighborhood  $V$ of $K$ such that $\ov V$ is compact in $U$. Then
\begin{eqnarray*} 
     \|1_K \mu\|&=&\int_K G\mu_V\,d\mu\le\int \int G(x,y)\,d\mu_V(y)\,d\mu(x)\\
       &\le &c^2\int\int G(y,x)\,d\mu(x)\,d\mu_V(y)=c^2\int
       G\mu(y)\,d\mu_V(y)\le c^2 \|\mu_V\|.
\end{eqnarray*} 
\end{proof} 

For $x\in X$ and $0<r<t$, we define balls $B(x,r)$ and shells $S(x,r,t)$ by
\[
                     B(x,r):=\{y\in X\colon \rho(x,y)<r\} , \quad     S(x,r,t):=\{y\in X\colon r\le \rho(x,y)<t\}\,.
\]
 Let us immediately note some elementary properties of  $R_1^{B(x,r)}$ and $\kap B(x,r)$.

\begin{proposition}\label{Rge}
Let $x\in X$, $r>0$ and $B: =B(x,r)$. Then  $R_1^B$ is a potential which is $\Px$-bounded,
\begin{equation}\label{rbg-1}
   R_1^B \le c\,\frac {G(\cdot,x)}{g(r)} \le c^2 \,\frac{g(\rho(\cdot,x))}{g(r)} , \qquad  
\|\mu_B\|\vee \kap B \le c g(r)\inv, 
\end{equation} 
\begin{equation}\label{rbg-2}
 R_1^B \ge  c\inv  \kap B \cdot g(\rho(\cdot,x)+r).
\end{equation} 
\end{proposition}

\begin{proof} We know that  $R_1^B\in \W$ (see  (\ref{red-W})). 
Moreover,  $G(\cdot,x)$ is a potential and $G(\cdot,x)\ge c\inv g(r)$ on $B$. Hence
  $R_1^B\le  \min\{1,  c  G(\cdot,x)/g(r)\} \in \Px$. In particular, $R_1^B$~is a~potential.
  
 Moreover,  let $\mu\in \M(X)$ such that $\mu(X\setminus \ov B)=0$ and $\int G(\cdot,z)\,d\mu(z)=G\mu\le 1$.
Since  $c\inv g(r)\le G(x, \cdot)$ on $\ov B$,   we see that $c\inv g(r) \|\mu\|\le 1$.

If even $\mu(X\setminus B)=0$, then, by the minimum principle (see \cite[III.6.6]{BH}),  
$R_1^B\ge G\mu$.  Let $y\in X$. For all $z\in B$, $\rho(z,y) < \rho(y,x)+ r$, and 
hence $G(y,z)\ge c\inv g(\rho(y,x)+r))$. Thus $R_1^B(y) \ge G\mu(y)\ge c\inv  g(\rho(y, x)+r) \|\mu\|$.
\end{proof}

\begin{assumption}\label{ass-2}
From now on we assume, in addition, the following.
\begin{itemize} 
\item[\rm (iv)]  \emph{Doubling property}: 
There exist $c_D\ge 1$ and $0\le R_0<\infty$ such that 
\begin{equation}\label{doubling}
     g(r/2)\le c_D g(r) \qquad\mbox{ for every }r>R_0.
\end{equation} 
\item[\rm (v)] All balls $B(x,r)$, $x\in X$, $r>0$, are relatively compact. 
\end{itemize} 
\end{assumption}

\begin{remarks}{\rm
1. We note that Assumptions \ref{ass-1} and \ref{ass-2}  are satisfied by rather general isotropic L\'evy processes 
(often with $R_0=0$; see \cite{grzywny} and \cite{H-fuku} for details). 

2. If (\ref{doubling}) is known to hold  with some $R_0>0$, then we may   replace
$R_0$ by any $R_0'\in (0,R_0)$ (at the expense of taking a larger $c_D'$), since $0<R_0'<r\le R_0$
implies that $g(r/2)\le g(R_0'/2)\le  g(R_0'/2) g(R_0)\inv g(r)$.

3. Since the function  $1$ is harmonic, $X$ cannot be compact, and hence (v) implies
that balls are proper subsets of $X$. 
}
\end{remarks}

Having the doubling property for $g$,  there  is a close relation between  estimates which are reverse to the ones in
(\ref{rbg-1}).

\begin{proposition}\label{equiv}
Let $x\in X$, $r>R_0$, $B:=B(x,r)$,  $C\ge 1$, and $\tilde C:=c^2c_DC$. Then
the following hold.
\begin{itemize}
\item[\rm (a)] 
If   $\kap B \ge C\inv g(r)\inv$, then  $R_1^B\ge \tilde C\inv G(\cdot,x)/g(r)$ on $X\setminus B$.
\item[\rm (b)]
If $R_1^B\ge C\inv G(\cdot,x)/g(r)$ on $X\setminus B$, 
then $\kap B\ge \tilde C\inv g(r)\inv $.
\end{itemize} 
\end{proposition}

\begin{proof} 
(a) Immediate consequence of (\ref{rbg-2}): it suffices to note
that, for every $y\in   B^c$, $\rho(y,x)+r\le 2 \rho(y,x)$ and hence $g(\rho(y,x)+r)\ge c_D\inv g(\rho(y,x))
\ge (cc_D)\inv G(y,x)$.

(b) Let $z\in X\setminus B(x,2r)$, $a:=\rho(x,z)$, and $\ve>0$. Then there exists 
   $0<r'<r$ such that $V:=B(x,r')$ satisfies $R_1^V(z)+\ve >R_1^B(z)\ge (cC)\inv g(a)/g(r)$. 
Moreover, 
\[
      R_1^V (z) =\int G(z,y)\,d\mu_V (y)\le  c g(a/2)  \|\mu_V\|\le cc_D g(a)\kap B,
\] 
since $a/2\le a-r< \rho(z,\cdot)$ on $B$.
Thus $\tilde C\inv  g(r)\inv \le \kap B$.
\end{proof}

Our main theorems  are the following. 

\begin{theorem}[Liouville property] 
Every function in $\tilde \H^+(X)$ is constant.
\end{theorem} 

\begin{theorem}[Wiener's test]\label{wiener-intro}
Let $A$ be a subset of $ X$, $x_0\in X$, $R>0$,  $\g>1$. Then
$A$ is unavoidable if and only if 
\begin{equation}\label{test}
\sum\nolimits_{n\in\nat} g(\g^n R) \kapo (A\cap S(x_0,\g^n R, \g^{n+1} R))=\infty.
\end{equation}
\end{theorem}

For the next two corollaries we suppose, in addition, 
that we have a measure $\lambda\in \M(X)$  with $\supp(\lambda)=X$  and such that, for some   $ c_0\ge 1$,  
the  normalized restrictions $\lambda_{B(x,r)}:=(\lambda (B(x,r)))\inv 1_{B(x,r)} \lambda$ of~$\lambda$ on~$B(x,r)$,
$x\in X$, $ r>R_0$,  satisfy
\begin{equation}\label{glr}
        G\lambda_{B(x,r)} \le c_0    g(r) 
\end{equation} 
so that, in particular,
\begin{equation}\label{kap-first}
                  \kap B(x,r)\ge c_0\inv g(r)\inv
\end{equation} 
(for many L\'evy processes, the Lebesgue measure will have this property; see \cite {H-fuku}). 

\begin{corollary}\label{balls-appl}
Let $A$ be a~union of pairwise disjoint balls $B(z,r_z)$, $z\in Z$, 
where  $Z\subset X$ is   locally finite  and  $r_z>4R_0$, and let $x_0\in X\setminus Z$ such that 
\[
     \inf\nolimits_{z,z'\in Z, \, z\ne z'} \frac {\lambda (B(z,\rho(z,z')/4))}{\lambda (B(x_0,4\rho(x_0,z)))} \cdot 
         \frac{g(r_z)}{g(\rho(x_0,z))} \, >\, 0.
\]
  Then $A$ is unavoidable if and only if
$\sum\nolimits_{z\in Z} g(\rho(x_0,z))/g(r_z)=\infty$.
\end{corollary}

\begin{definition}\label{reg-loc} 
We shall say that pairwise disjoint balls  $B(z,r_z)$,
$z\in Z$,   are \emph{regularly located} if the following hold:
\begin{itemize} 
\item There exists $\ve>0$ such that $\rho(z,z')\ge \ve$, for all $z,z'\in Z$, $z\ne z'$.  
\item There exists $R>0$ such that every ball of radius $R$ contains a
  point of $Z$.
\item There exists a decreasing function $\phi\colon (0,\infty)\to
  (4R_0,\infty)$   and $C> 1$ 
such that
\begin{equation}\label{radius-phi}
\phi(\rho(x_0,z))<   r_z < C \phi(\rho(x_0,z)),     \qquad z\in Z.
\end{equation} 
\end{itemize} 
 \end{definition}

Under mild additional assumptions on $\lambda$ (see Section \ref{distr}), which are satisfied if $X=\reald$, $\rho$
is the Euclidean metric and $\lambda$ is Lebesgue measure, the following holds. 

\begin{corollary}\label{corollary-sep}
Let $A$ be  a  union  of balls   $B(z,r_z)$, $z\in Z$,   in $X$
which are regularly located. Then $A$  is  unavoidable if and only if 
 $\sum\nolimits_{z\in Z} g(\rho(x_0,z))/ g(r_z) =\infty$.
\end{corollary}

\section{Liouville property}

\begin{proposition}\label{har-first}
Let $U$ be an open set in $X$ and $B:=B(x_0,R)$,  $x_0\in U$, $R>R_0$,
 such that the closure of $B':=B(x_0,2R)$ is contained in
$U$. Then
\begin{equation}\label{h-global}
   \sup h(B)\le (cc_D)^2 \inf h(B) \qquad\mbox{for every }  h\in \W\cap \tilde \H^+(U).
\end{equation} 
\end{proposition} 

\begin{proof} We may choose 
  $\vp_n\in\K^+(X)$, $n\in\nat$, such that $\vp_n\uparrow  1_{X\setminus \ov{B'}}h$. 
Then
\[
           h_n:=R_{\vp_n } \in \Px\cap \H^+(X\setminus \supp(\vp_n)), \qquad n\in\nat,
\]
by \cite[III.2.4, III.5.6]{BH}. So there exist   measures~$\mu_n$ on~$X$, $n\in\nat$,  such that 
\[
             h_n=G\mu_n \und \mu_n(B')=0. 
\]
For all  $x,y\in B$, and $z\in X\setminus B'$, 
$\rho(x,z)<R+\rho(x_0,z)\le  
 3(\rho(x_0,z)-R)<3\rho(y,z)
$,
hence, defining $K:=(cc_D)^2$,  $G(x,z)\le c g(\rho(x,z))\le cc_D^2 g(\rho(y,z))\le K G(y,z)$ and 
\begin{equation}\label{hnhn}
        h_n(x)=\int G(x,z)\,d\mu_n(z)\le K\int G(y,z)\,d\mu_n(z)       
                \le K h_n(y)
\end{equation} 
for every $n\in\nat$. The sequence $(h_n)$ is increasing to a function $u\in \W$ which, by~(\ref{hnhn}), 
satisfies  $\sup u(B)\le K\inf u(B)$. Clearly, $u\le h$ on $X$ and $u=h$ on $X\setminus \ov B'$.

Let $V$ be an open neighborhood of $\ov {B'}$ such that
  $\ov V$ is compact in $U$. Then $u\ge R_h^\vc$, since $\vc\subset X\setminus \ov B'$, and,
 for every $x\in V$,  $u(x)\ge \vx^\vc(h)=h(x)$. Thus $u=h$ completing the proof.
% Moreover, $u:=\sup_{n\in\nat} h_n\in \W$, $u\le h$ on $X$  and  $u=h$ on~$\vc$. 
% Hence $u\ge R_h^\vc$  and, for every $x\in V$,
% \[
%        h(x)\ge  u(x)\ge \vx^\vc (h)=h(x).
% \]
%  In particular, $\sup h(B)=\sup u(B)\le K\inf u(B)=K \inf h(B)$, by (\ref{hnhn}). 
\end{proof} 

\begin{corollary}[Liouville property]\label{global-h}
Every function in $\tilde \H^+(X)$ is constant. 
\end{corollary} 

\begin{proof} 

For every relatively compact open set $V$ in $X$ and every  $f\in \B^+(X)$, the function  $H_V f$ is lower semicontinuous 
on $V$ (see \cite[III.3.4]{BH}). 
Therefore $\tilde \H^+(X) \subset \W$ (see, for example, \cite[II.5.5]{BH}), and we obtain that  (\ref{h-global}) holds
for all $h\in \tilde \H^+(X)$ and  balls $B(x_0,r)$, $x_0\in X$, $r>R_0$.

Now the claim follows immediately by a well known standard argument.
(Let $h\in\tilde \H^+(X)$ and $a:=\inf h(X)$. Then  $h':=h-a\in \tilde
\H^+(X)$. Given  $\ve>0$, there exists
$x_0\in X$ such that $h'(x_0)<\ve$ and hence, considering $x\in X$ and
$R>R_0\vee \rho(x,x_0)$, we obtain that $h'(x)\le (cc_D)^2 h'(x_0)<
(c c_D)^2\ve$. Thus $h'=0$.)
\end{proof} 

\section{Proof of Wiener's test} 

One direction of  Wiener's test  is an easy consequence of Proposition \ref{simple}(b).    
We only have to use the definition of $\kapo$ and note the simple fact that, for every open set~$V$ 
which is contained in an open ball, the reduced function $R_1^V$ is a potential, by  Proposition  \ref{Rge}. 

\begin{proposition}\label{only-if}   
Let $A$ be an unavoidable set in $X$, $x_0\in X$, $R>0$, $\g>1$. Then
\[
\sum\nolimits_{n\in\nat} g(\g^n R) \kapo (A\cap S(x_0, \g^n R, \g^{n+1} R))=\infty.
\]
\end{proposition} 

\begin{proof} 
For   $n\in\nat$, there are open neighborhoods~$U_n$ of 
 $ A_n:=A\cap S(x_0, \g^n R, \g^{n+1} R)$  in~$S(x_0,\g^n R/2, \g^{n+1} R)$ 
such that
\begin{equation}\label{kapiU}
         \kap U_n\le \kapo A_n+2^{-n}.
\end{equation} 
Since $A\setminus B(x_0,\g R) \subset \bigcup_{n\in\nat} U_n$, we know, by Proposition \ref{simple},  that
\begin{equation}\label{runi}
         \sum\nolimits_{n\in\nat} R_1^{U_n} =\infty.
\end{equation} 
By  \cite[VI.1.7]{BH}, there  exist open sets $V_n$ such that $\ov V_n$ is compact in $U_n$ and 
\[
              R_1^{U_n}(x_0)\le R_1^{V_n}(x_0)+2^{-n}, \qquad n\in\nat. 
\]
Then, by (\ref{runi}), 
\[
        \sum\nolimits_{n\in\nat} R_1^{V_n}(x_0)=\infty.
\] 
Let $n_0\in \nat$ such that $\g^{n_0} R/2>R_0$. For  $n\ge n_0$, let $\nu_n:=\mu_{V_n}$, that is, $G\nu_n=R_1^{V_n}$.
Since  $\nu_n$ is supported by the set $\ov V_n$,
which does not intersect $B(x_0, \g^n R/2)$, and 
\[
G(x_0,\cdot)\le c g(\rho(x_0,\cdot))\le c g(\rho(\g^n R/2))\le c c_D g(\g^n R) \on X\setminus B(x_0,\g^n R/2),
\]
we see that $ R_1^{V_n}(x_0)=\int G(x_0,y) d\nu_n(y)\le c c_D g(\g^n R) \|\nu_n\|\le c c_D g(\g^n R) \kap U_n$. Therefore
\[
        \infty=\sum\nolimits_{n\ge n_0} R_1^{V_n}(x_0) \le \sum\nolimits_{n\ge n_0} c c_D g(\g^n R)\kap U_n.
\]
Since $g(\g^nR)\le g(R)<\infty$, for every $n\in\nat$, we finally conclude from (\ref{kapiU}) that 
 $\sum_{n\in\nat} g(\g^n R) \kapo A_n=\infty$.
\end{proof} 

Knowing that positive harmonic functions on $X$ are constant, by Corollary \ref{global-h}, 
a set $A$ in $X$ is avoidable if and only if it is minimally thin at infinity (see \cite[Proposition 2.3]{HN-unavoidable}). 
Therefore it suffices to modify the proofs for \cite[V.4.15 and V.4.17]{BH}  (characterizing, 
in the setting of Riesz potentials, thinness of a set $A$ at a~point). 
In the context of L\'evy processes,  this has already been noted (see,
for example, \cite[Proposition 7.3 and Corollary 7.4]{mimica-vondracek}). 
In our situation, the zero-one law will yield a~straight forward
modification.

\begin{theorem}\label{V.4.15} 
Let    $A\subset X$, $x_0\in X$,   $s_n\in (0,\infty) $ and $\delta\in (0,1) $ with
 $s_n\le \delta s_{n+1}$ for every $n\in\nat$.    
Then the following hold for the sets  $ A_n:= A\cap S(x_0,s_n , s_{n+1}  )${\rm:}
\begin{itemize}
\item[\rm (i)] If $A$ is  unavoidable, then $\sum_{n\in\nat} R_1^{A_n}=\infty$ on $X$.
\item[\rm (ii)] If $A$ is avoidable, then $\sum_{n\in\nat} R_1^{A_n}<\infty$ on $X$.
\end{itemize} 
\end{theorem} 

\begin{proof} 
(i) Proposition  \ref{simple}, (b).

(ii) By Proposition \ref{so-simple}, there exists a point $x_1\in X$ such that 
$                                                     R_1^A(x_1)< (c^2 c_D)\inv$.
By \cite[VI.1.2]{BH}, there exists an open neighborhood $V$ of $A$ such that 
\begin{equation}\label{def-a}
 a:=c^2c_D R_1^{V}(x_1)<1.
\end{equation} 
The Liouville property implies that $R_1^V$ is a potential (see  \cite[Proposition 2.3]{HN-unavoidable}).
Let
\[
          V_n:=V\cap S(x_1,s_n,s_{n+3}),  \qquad n\in\nat.
\]
Since $A_n\subset V_{n-1}$ if $n$ is sufficiently large, it suffices
to show that 
\begin{equation}\label{first-claim} 
\sum\nolimits_{n\in\nat} R_1^{V_n}<\infty.
\end{equation} 
To prove (\ref{first-claim})  let 
$k\in\nat$ such that $1-\delta^{k-3}>1/2$. 
We fix $1\le i\le k$ and define 
\[ 
         U_n:=V_{i+nk}, \qquad n\in\nat.
\]
It clearly suffices  to show that  
$\sum_{n\in\nat} R_1^{U_n}<\infty$. Let
$U:=\bigcup\nolimits_{n\in\nat}  U_n$. Since $R_1^U\le R_1^V$, $R_1^U$
is a potential as well, $R_1^U=G\mu_U$. For $n\in\nat$, let
\[
  \mu_n:=1_{\ov U_n} \mu_U
\]
 so that $\sum_{n\in\nat} G\mu_n=G\mu_U=R_1^U\le 1$.

Let $n_0\in\nat$ such that $s_{n_0}>R_0$. For the moment, let us fix
  $n\ge n_0$, consider $m\in\nat$, $m\ne n$, $y\in U_n$, and $z\in \ov U_m$.
If~$m<n$, then $\rho(y_1,z)\le \delta^{k-3}\rho(x_1,y)$. If $m>n$, then $\rho(x_1,y)\le \delta^{k-3} \rho(x_1,z)$.
In~both cases, 
\[
\rho(y,z)\ge (1-\delta^{k-3})\rho(x_1,z)\ge \rho(x_1,z)/2.
\] 
 Defining $\mu_n':=\mu_U-\mu_n=\sum_{m\ne n} \mu_m$ we hence obtain that, for every $y\in U_n$, 
\[
   G\mu_n'(y)\le c \int g(\rho(y,z))\,d\mu_n'(z)\le c c_D\int g(\rho(x_1,z))\,d\mu_n'(z)
\le c^2 c_D G\mu_n'(x_1) \le a. 
\]
Since $G\mu_n+G\mu_n'=G\mu_U=R_1^U$ and $R_1^U=1$ on $U$, we therefore conclude that 
$1-a\le G\mu_n$ on $U_n$, and hence
\[
(1-a)R_1^{U_n}\le G\mu_n.
\] 
Thus $\sum_{n\ge n_0} R_1^{U_n}\le (1-a)\inv \sum_{n\ge n_0}G\mu_n\le
(1-a)\inv G\mu_U\le (1-a)\inv$ completing the proof.
\end{proof}

\begin{proof}[Proof of Theorem \ref{wiener-intro}]
Let $A\subset X$, $x_0\in X$, $R>0$ and $\g >1$.

If $A$ is unavoidable, then (\ref{test}) holds, by Proposition \ref{only-if}.   

So let us assume that $A$ is avoidable. 
By \cite[VI.1.5]{BH}, there exists an open neighborhood $U$ of $A$ which is avoidable. 
For every $n\in\nat$, 
\[
     U_n:=U\cap S(x_0,  \g^{n-1} R, \g^{n+1}R)
\]
is an open neighborhood of $A\cap S(x_0,\g^n R, \g^{n+1} R)$. 
By Proposition \ref{V.4.15}, 
\[
          \sum\nolimits _{n\in\nat} R_1^{U_n} (x_0)<\infty.
\]
By Lemma \ref{cap-potential},  there exist   open sets $V_n$ in $U_n$ such that $\ov V_n$ is compact in $U_n$ and 
\[
             \kap U_n\le c^2 \|\mu_{V_n}\|+2^{-n},\qquad n\in\nat.
\]
Let $k\in\nat$ such that $\g\le 2^k$, and let $n\in\nat$ such that $\g^n R>R_0$. Then
$g(\g^n R)\le c_D^k g(\g^{n+1} R)\le c_D^k g(\rho(x_0,\cdot))$ on $\ov V_n$, and hence 
\[
g(\g^n R) \|\mu_{ V_n}\|\le c_D^k \int g(\rho(x_0,y))\, d\mu_{V_n}(y)\le c c_D^k G\mu_{V_n} (x_0). 
\]
Since $G\mu_{V_n}=R_1^{V_n} \le R_1^{U_n}$ and  $g(\g^n R)\le g(R)$, $n\in\nat$, we conclude that 
\[
      \sum\nolimits_{n\in\nat} g(\g^n R) \kap U_n <\infty.
\]
Thus $\sum\nolimits_{n\in\nat} g(\g^n R) \kapo (A\cap S(x_0,\g^n R, \g^{n+1} R))<\infty$.
\end{proof}

\section{Application to collections of balls having the separation property}\label{appl-ball}

\begin{corollary}\label{ball-nec} 
Let $B(z,r_z)$, $z\in Z\subset X$, $r_z>0$, be balls in $X$ such that their union $A$ is 
unavoidable. Then, for every $x_0\in X$,
\[
   \sum\nolimits_{z\in Z} g(\rho(x_0,z))/g(r_z)=\infty.
\]
\end{corollary} 

\begin{proof} Propositions \ref{simple} and   \ref{Rge}.
\end{proof}

The next simple result on  comparison of potentials (cf.\ \cite{H-fuku})  will be sufficient for us
(see the proof of \cite[Theorem 5.3]{HN-unavoidable} for a~much more
delicate version; cf.\ also the proof of \cite[Theorem 3]{aikawa-borichev}).

\begin{lemma}\label{comparison}
Let $Z\subset \reald$ be finite and $r_z>R_0$, $z\in Z$,  such that, for $z\ne z'$,
 $B(z,r_z)\cap B(z',3r_{z'})=\emptyset$.  Let $w\in  \W\cap \C(X)$ and, for every $z\in Z$, let 
$\mu_z,\nu_z$ be measures on~$ B(z,r_z)$ such that  $G\mu_z\le w$, and $\|\mu_z\|\le \|\nu_z\|$.
Then  $\mu:=\sum_{z\in Z} \mu_z$ and $\nu:=\sum_{z\in Z}\nu_z$ satisfy 
 \begin{equation}\label{comp-mu-nu}
             G\mu\le  w+  c^2c_D G\nu.
 \end{equation} 
\end{lemma} 

\begin{proof} 
Let $z,z'\in Z$, $z'\ne z$, and $x\in   B(z,r_z)$. For all $y,y'\in   B(z',r_{z'})$,  $\rho(y,y')<2r_{z'}<  \rho(x,y)$,
 hence $ \rho(x,y')\le  2\rho(x,y) $, $g (\rho(x,y))\le c_D g(2\rho(x,y))\le c_D g(\rho(x,y'))$, and $G(x,y)\le c^2c_D G(x,y')$. By integration, 
$G\mu_{z'} (x)\le c^2c_D G\nu_{z'}(x)$. Therefore
\begin{center}
$        G\mu(x)= G\mu_z(x)+\sum\nolimits_{z'\in Z, z'\ne z} G\mu_{z'}(x)\le w(x)+c^2c_D G\nu(x)$.
\end{center}
Thus $G\mu\le w+c^2 c_D G\nu$ on the union $A$ of the balls $  B(z,r_z)$, $z\in Z$.  By the  minimum principle (see \cite[III.6.6]{BH}),
the proof is finished.
\end{proof}

\begin{lemma}\label{shell-1}
Let $x_0\in X$, $R>2R_0$ and  $B:=B(x_0,R)$. Suppose that there exist $C\ge 1$ and a probability measure $\lambda$
on $B$ such that $G\lambda\le C g(R)$.
Let    $Z$ be a~finite subset of~$B(x_0,R/2)$  and $ R_0<r_z\le R/2$, $z\in Z$,  such that 
the balls~$B(z, 3  r_z)$ are pairwise disjoint and, for some $\ve\in (0,1)$,
\begin{equation}\label{sep}
g(r_z)   \lambda(B(z,\rho(z,z')/4 ))      \, \ge \, \ve g(R)  
,\qquad\mbox{ whenever }   z\ne z'.
\end{equation} 
 Then the union $A$   of the balls $B(z,r_z)$, $z\in Z$, satisfies 
\[
\kap A\ge \ve (2c^3c_DC)\inv    \sum\nolimits_{z\in Z} \kap B(z,r_z). 
\]
\end{lemma}

\begin{proof} It clearly suffices to consider the case, where $Z$ contains more than one point.
Then, for $z\in Z$,
\begin{equation}\label{def-tilde-r}
     \tilde r_z:=   \max\bigl\{r_z, \dist(z, Z\setminus \{z\})/4\bigr\}\le R/2, 
\end{equation} 
  hence  $B(z,\tilde r_z)\subset B$ and $\lambda(B(z,\tilde r_z))>0$, by (\ref{sep}).
Further,  $ B(z,\tilde r_z)\cap B(z',3\tilde r_{z'})=\emptyset$, whenever  $ z\ne z'$,

For  $z\in Z$, let $\mu_z\in \M(X)$ with $\mu_z(X\setminus B(z,r_z))=0$ and $G\mu_z\le 1$, and let
\[
\a_z:=\|\mu_z\|/\lambda(B(z,\tilde r_z)), \qquad \nu_z:=\a_z1_{B(z,\tilde r_z)} \lambda.
\] 
 Then $\|\nu_z\|=\|\mu_z\|$ and,  by  Proposition \ref{Rge},  (\ref{def-tilde-r}), and (\ref{sep}), 
\[
\a_z\le c g(r_z)\inv/\lambda(B(z,\tilde r_z))\le  c (\ve g(R))\inv.
\]

Since the balls $B(z,\tilde r_z)$, $z\in Z$,  are pairwise disjoint subsets of~$B$,  
the measure  $\nu:=\sum_{z\in Z} \nu_z$ satisfies
\[
                        G\nu \le  c (\ve  g(R))\inv  G\lambda\le cC\ve\inv.
\] 
Let  $\mu:=\sum_{z\in Z} \mu_z$. By  Lemma \ref{comparison},    $ G\mu\le 1+c^2 c_D G\nu$. 
Thus $G\mu\le 2 c^3c_DC\ve\inv$. 
Since $\mu(X\setminus A)=0$,  we see that
$
          \kap A\ge \ve (2c^3c_DC)\inv \sum\nolimits_{z\in Z}  \|\mu_z\|
$
completing the proof.
\end{proof}

In addition to the Assumptions \ref{ass-1} and \ref{ass-2},  we suppose   the following.

\begin{assumption}\label{ass-coll}
We have a measure $\lambda\in \M(X)$  and a constant $c_0>0$ such that  $\supp(\lambda)=X$ 
and, for all  $x\in X$ and $r>R_0$, 
\begin{equation}\label{ass-for-sep}
            G\lambda_{B(x,r)} \le  c_0 g(r)
\end{equation} 
{\rm(}where, as before, $\lambda_B:=\lambda(B)\inv 1_B \lambda$ for every ball $B${\rm)}. 
\end{assumption} 

We already observed that then,  for all $x\in X$ and $r>R_0$, 
\begin{equation}\label{cap-lower}
                  \kap B(x,r)\ge c_0\inv g(r)\inv
\end{equation} 
(see Proposition \ref{Rge} for the corresponding upper estimate). By Proposition \ref{equiv},  
such a lower estimate for the capacity of balls is equivalent
to  having inequalities $R_1^{B(x,r)}\ge C\inv G(\cdot,x)/g(r)$ on $X\setminus B(x,r)$
(which in turn, by the minimum principle,  hold  trivially if $(X,\W)$ is a harmonic space, that is, if $\mathfrak X$
has no jumps). 

 Let us say that a family of pairwise disjoint balls $B(z,r_z)$, $z\in Z\subset X$, $r_z>2 R_0$, 
has the \emph{separation property with respect to $\lambda$}, if $Z$ is locally finite
    and, for some point~$x_0\in X\setminus Z$,  
\begin{equation}\label{separation}
     \inf\nolimits_{z,z'\in Z, z\ne z'} \frac {\lambda (B(z,\rho(z,z')/4))}{\lambda (B(x_0,4\rho(x_0,z)))} \cdot 
         \frac{g(r_z)}{g(\rho(x_0,z))} \, >\, 0.
\end{equation} 

\begin{remark}{\rm
If, for example,  $\lambda$ is Lebesgue measure on $X=\reald$, $x_0=0$, $\rho(x,y)=|x-y|$, and $g(r)=r^{\a-d}$, then (\ref{separation})
means that 
\[
     \inf\nolimits_{z,z'\in Z, z\ne z'} \frac {|z-z'|^d}{|z|^\a r_z^{d-\a}}\,>\,0,
\]
which, in the classical case $\a=2$,  is the separation property in \cite[Theorem 6]{gardiner-ghergu}.
}
\end{remark}

\begin{theorem}\label{main-sep} Let $A$ be an avoidable union of pairwise disjoint balls $B(z,r_z)$, $z\in Z\subset X$, 
having the separation property with respect to $\lambda$.
 Then 
\[
\sum\nolimits_{z\in Z} g(\rho(x_0,z)) \kap B(z,r_z)<\infty.
\]
\end{theorem}

\begin{proof} We may suppose    that $\rho(x_0,z)>4R_0$, for every $z\in Z$  (we simply omit finitely many points from $Z$). 
Moreover, we may assume without loss of generality that 
\begin{equation}\label{rz}
                          r_z\le \rho(x_0,z)/2, \qquad\mbox{ for every }z\in Z.
\end{equation} 
Indeed, replacing $r_z$ by  
$
r_z':=\min\{r_z,\rho(x_0,z)/2\}
$
our assumptions are  preserved.
Suppose we have shown that   $\sum_{z\in Z} {g(\rho(x_0,z))} /{g(r_z')}<\infty$.
Since $g(r)/g(r/2)\ge c_D\inv$ if $r>2R_0$, the set $Z'$ of all points $z\in Z$ such that 
$r_z'=\rho(x_0,z)/2$ is finite, and therefore
$\sum_{z\in Z'} g(\rho(x_0,z))/g(r_z)<\infty$. So we may assume without loss of generality 
that $r_z'=r_z$, for all $z\in Z$, that is,  (\ref{rz}) holds.

Further, we may assume   that the balls $B(z, 4r_z)$  are pairwise
disjoint. Indeed,   since  $\kap B(z,r_z)\approx g(r_z)\inv$ and 
$g(r)\le g(r/4) \le c_D^2 g(r)$, $r>2R_0$,
 a replacement of~$r_z$ by~$r_z/4$ does neither affect~(\ref{separation}) nor the convergence or divergence 
of $\sum_{z\in Z} {g(\rho(x_0,z))} \kap B(z,r_z)$, 
 and the new, smaller union is, of course, avoidable.

 By (\ref{separation}),   there exists $\ve\in (0,1)$  such that,  for  $z,z'\in Z$,  $ z\ne z'$,
\begin{equation}\label{strong-sep}
g(r_z) \lambda(B(z,\rho(z,z')/4))  \ge  \ve  g(\rho(x_0,z)) \lambda(B(x_0,4 \rho(x_0,z))) . 
\end{equation} 

Let $R> R_0$. For  $n\in\nat$, let
\[
        Z_n:=Z\cap S(x_0,2\cdot 8^n R, 4\cdot 8^n R) \und A_n:=\bigcup\nolimits_{z\in Z_n} B(z,r_z).
\]
Then $A_n\subset   A\cap S(x_0,8^nR,8^{n+1}R)$. Moreover, for every $z\in Z_n$, $\rho(x_0,z)\ge 8^n R$,
and hence $g(\rho(x_0,z))\le g(8^n R)$. Therefore, by Lemma \ref{shell-1} and Theorem \ref{wiener-intro},
\begin{multline*} 
    \sum\nolimits_{n\in\nat} \sum\nolimits_{z\in Z_n} g(\rho(x_0,z))\kap B(z,r_z) \\
\le 2c^3c_Dc_0\ve\inv 
\sum\nolimits_{n\in\nat} g(8^n R)\kapo (A\cap S(x_0, 8^n R, 8^{n+1} R))<\infty.
\end{multline*} 
Applying this estimate as well to $2R$ and $4R$ in place of $R$ we obtain that
\[
  \sum\nolimits_{z\in Z} g(\rho(x_0,z))\kap B(z,r_z)<\infty.
\]
\end{proof}

\begin{corollary}\label{final-corollary}
For every union $A$ of  balls $B(z,r_z)$, $z\in Z\subset X$, 
having the separation property with respect to $\lambda$, the following statements are equivalent.
\begin{itemize} 
\item[\rm (1)] The set $A$ is unavoidable.
\item[\rm (2)] $\sum\nolimits_{z\in Z} g(\rho(x_0,z)) \kap B(z,r_z)=\infty$.
\item [\rm (3)] $\sum\nolimits_{z\in Z} g(\rho(x_0,z)) /g(r_z) =\infty$.
\end{itemize}
\end{corollary}

\begin{proof} 
Corollary \ref{ball-nec} and Theorem \ref{main-sep} using   $\kap B(z,r_z) \approx  g(r_z)\inv$, $z\in Z$.
\end{proof}

\section{Application to regularly located balls}\label{distr}

In this section we suppose as before that  the Assumptions~\ref{ass-1},   \ref{ass-2},  and \ref{ass-coll} are satisfied. 
  Moreover, let us assume 
that we have a~distinguished point $x_0\in X$  such that  the  measure~$\lambda$
has the following additional properties:
\begin{itemize}
\item [\rm (i)]
For all  $x,y\in X$ and  $r>R_0$,
\begin{equation}\label{trans}
  \lambda(B(y,r))\le c_0\lambda (B(x,r)). 
\end{equation} 
\item [\rm (ii)]
There exist $C_D\in (1,\infty)$, $\kappa\in (0,1)$   such that, for all $r>R_0$,
\begin{equation}\label{double-ball-shell}
   \lambda(B(x_0,2r))\le C_D \lambda(S(x_0,\kappa r, r)).
\end{equation}
\end{itemize} 

Let $A$ be a union of balls $B(z,r_z)$, $z\in Z$, which is regularly located 
(see \ref{reg-loc}). We first prove the following proposition.

\begin{proposition}\label{MV}
Suppose that 
$\limsup_{r\to \infty} \lambda(B(x_0,r))g(r)/g(\phi(r))>0$.
 Then  $A$  is unavoidable.
\end{proposition} 

\begin{proof} We may assume without loss of generality that the radius $R$
in (\ref{reg-loc}) satisfies $R>  \max\{1,\phi(1),R_0\}$.
We define 
\[
        a:=(cc_0^2 C_D^5 \lambda(B(x_0, R))\inv \und b:=  cc_0^2\lambda(B(x_0,  R))\inv.
\]
Let $0<\b<\limsup_{r\to \infty} \lambda(B(x_0,r))g(r)/g(\phi(r))$. 

We now  fix $x\in X$ and choose $  r>\kappa\inv( 4R +2\rho(x_0,x)) $
such that $r>\phi(r)$ and 
\begin{equation}\label{rho-choice}
\g:=\lambda(B(x_0,r)) g(r)/g( \phi(r))> \b. 
\end{equation} 
Let
\[
   S:=S(x_0,\kappa r/2, r/2), \qquad
      B:=B(x_0,r) \und  r_0:=  \phi(r)
\]
so that 
\[
        \g = \lambda(B) g(r)/g(r_0).
\]

There are finitely many points $y_1,\dots,y_m\in \ov S$ 
 such that  
$B(y_1,3R),\dots, B(y_m,3R)$ are pairwise disjoint and
$\ov S$ is covered by $B(y_1,9R), \dots, B(y_m,9R)$.
We may choose  $z_j\in Z\cap B(y_j,R)$, $1\le j\le m$. Then $\rho(z_i, z_j)\ge \rho(y_i,y_j)-2R\ge 4R$,
and hence
\begin{equation}\label{distance} 
B(z_i,R)\cap B(z_j,3R)=\emptyset,
\end{equation} 
for all $i,j\in \{1,\dots, m\}$ with $i\ne j$. Moreover,       
\begin{equation}\label{m-lower} 
\lambda(B)\le C_D\lambda(S)\le C_D \sum\nolimits_{j=1}^m \lambda
(B(y_j,9R))\le m C_D^5 c_0 \lambda(B(x_0,R)).  
\end{equation} 
 
Let $1\le j\le m$. Clearly,
\[
r>r-R> r/2+R\ge \rho(x_0,z_j) \ge \kappa r/2-R\ge R\ge 1.
\]
 Therefore $B(z_j,R)\subset B$ and $r_0=\phi(r)\le \phi(\rho(x_0,z_j))<r_{z_j}$,   hence $\ov B(z_j,r_0)\subset A$. Moreover,
$r_0 \le \phi(1)\le R$ and  $r_0+\rho(x,z_j) \le R+\rho(x_0,x)+ r/2+R\le  r$.
So $g(\rho(x,z_j)+r_0)\ge  g(r)$ and, by  (\ref{rbg-2}) and (\ref{cap-lower}), 
\begin{equation}\label{Rzj}
                 R_1^{B(z_j,   r_0)}(x)\ge (cc_0)\inv g(\rho(x,z_j)+r_0)/g(r_0) \ge   (cc_0)\inv g(r)/g(r_0).
\end{equation} 
Let $\mu_j$ be the equilibrium measure for $B(z_j,r_0)$, $1\le j\le m$. We define
\[
       p:=   \sum\nolimits_{j=1}^m R_1^{B(z_j,r_0)}=\sum\nolimits_{j=1}^m G\mu_j.
\]
Then, by (\ref{Rzj}) and (\ref{m-lower}), 
\begin{equation}\label{est-p}
  p(x) \ge m  (cc_0)\inv g(r)/g(r_0) \ge   a\g .
\end{equation} 

Finally, let $\nu:=\sum\nolimits_{j=1}^m \nu_j$, where 
\[
          \nu_j:=c g(r_0)\inv \lambda_{B(z_j,R)} \le cc_0g(r_0)\inv \frac{\lambda(B)}{\lambda(B(x_0,R))} \, 1_{B(z_j,R)} \lambda_B. 
\]
Since   $B(z_1,R), \dots, B(z_j, R)$ are pairwise disjoint subsets of~$B$ and 
$G\lambda_B\le  c_0 g(r)$, 
\[ 
       G\nu\le   cc_0 g(r_0)\inv  \frac{\lambda(B)}{\lambda(B(x_0,R))} \, G\lambda_B \le cc_0^2  \frac{\lambda(B)}{\lambda(B(x_0,R))}\, g(r)/g(r_0)
 = b\g.
\] 
By Proposition \ref{Rge}, $\|\mu_j\|\le c g(r_0)\inv$,  
 $1\le j\le m$. Thus, by (\ref{distance})  and Lemma \ref{comparison}, 
\[
                p    \le 1+c^2c_D G\nu\le 1+c^2c_D b\g.
\]
Since $\mu$ is supported by the   compact 
$\ov B(z_1,r_0)\cup\dots \cup \ov B(z_m,r_0)$ in $A$, this implies that 
\[
             R_1^A\ge  (1+c^2c_D b\g)\inv p, 
\]
by  the minimum principle (see ~\cite[III.6.6]{BH}). In particular, 
\[
           R_1^A(x)\ge  \frac{a\g}{1+c^2c_Db\g}=\frac a{\g\inv +c^2c_Db} > \frac a {\b\inv +c^2 c_D b}\, ,
\]
by (\ref{est-p}) and (\ref{rho-choice}).  Thus  $A$ is unavoidable, by Proposition  \ref{so-simple}. 
\end{proof}

\begin{proof}[Proof of Corollary \ref{corollary-sep}]
If $A$ is unavoidable, then $\sum_{z\in Z} g(\rho(x_0,z))/g(r_z)=\infty$, by Corollary \ref{ball-nec}.

To prove the converse, suppose that $\sum_{z\in Z} g(\rho(x_0,z))/g(r_z)=\infty$. 
By Proposition \ref{MV} and (\ref{double-ball-shell}), it suffices to consider the case
\begin{equation}\label{sufficient}
\lim\nolimits_{r\to\infty} \lambda(B(x_0,4 r)) g(r)/g(\phi(r))=0.
\end{equation} 
Then  $\inf\nolimits_{r>0} g(\phi(r)) \bigl(\lambda(B(x_0,4 r)) g(r)  )\bigr)\inv >0$. By (\ref{radius-phi}), 
$g(\phi(r_z))\approx g(\rho(x_0,z))$, $z\in Z$. Moreover,
 \[
 \lambda(B(z, \rho(z,z')/4))\ge \inf\nolimits_{x\in X} \lambda(B(x,\ve/4))>0,
\]
whenever, $z,z'\in Z$, $z\ne z'$. So the balls $B(z,r_z)$, $z\in Z$, have the separation property, and 
   $A$~is  unavoidable,  by Corollary \ref{final-corollary}.
\end{proof}

\section{Appendix}

The following equivalences are of independent interest and may be useful in applications.

\begin{proposition}\label{3G} Let $X$ be an arbitrary set and $G\colon X\times X\to [0,\infty]$
such that $G=\infty$ on the diagonal and $0<G<\infty$ outside the diagonal.
Then the  following properties are equivalent:
\begin{enumerate} 
\item[\rm (i)]  
$G$ has the \emph{triangle property}: There exists $C\ge 1$ such that 
\[ 
\min\{G(x,z),G(y,z)\}\le C G(x,y), \qquad x,y,z\in X.
\] 
\item[\rm (ii)] 
There exists a metric $\rho$ on $X$ and $\g>0$ such that 
$         G\approx \rho^{-\g}$. 
\item[\rm (iii)] 
There exist a metric $\rho$ on~$X$, a decreasing  function $g\colon [0,\infty)\to (0,\infty]$,
  $c_D\ge 1$, and  $\eta_0,\a_0\in (0,1)$ such that  $G\approx g\circ \rho$  and,  for every $r>0$, 
\[
g(r/2)\le c_D g(r)
\und     g(r)\le \eta_0 g(\a_0 r).
\]
\item[\rm (iv)]
There exist a metric $\rho$ on~$X$, a decreasing  function $g\colon [0,\infty)\to (0,\infty]$,
and $c_D\ge 1$  such that $G\approx g\circ \rho$  and $g(r/2)\le c_D g(r)$ for every $r>0$.
\end{enumerate} 
\end{proposition} 

\begin{proof} 
(i)$\,\Rightarrow\,$(ii): 
Since $G=\infty$ on the diagonal, the triangle property  implies that $G(y,x)\le C G(x,y)$
and 
$                              \tilde \rho(x,y):= G(x,y)\inv + G(y,x)\inv$, $x,y\in X$,
defines a quasi-metric on $X$ which is equivalent to $G\inv$. By  \cite[Proposition 14.5]{heinonen}
(see also \cite[pp.\ 1209--1212]{convexity} and  \cite{H-uniform}), there exists a metric $\rho$ for $X$
and $\g>0$ such that $\tilde \rho \approx \rho^{\g}$, and hence $G\approx \rho^{-\g}$.

(ii)$\,\Rightarrow\,$(iii): 
Trivial defining $g(r):=r^{-\g}$.

(iii)$\,\Rightarrow\,$(iv): Trivial.

(iv)$\,\Rightarrow\,$(i): 
Let $c>0$ such that $ c\inv  g\!\circ\! \rho \le G\le cg\!\circ\! \rho$ and   let
$x,y,z\in X$. Since $\rho(x,y)\le \rho(x,z)+\rho(y,z)$, we know that $\rho(x,z)\ge \rho(x,y)/2$
or $\rho(y,z)\ge \rho(x,y)/2$. Therefore 
\begin{eqnarray*} 
           \min\{G(x,z),G(y,z)\}&\le& c\min\{g(\rho(x,z)),g(\rho(y,z))\}\\
                                           &\le &c g(\rho(x,y)/2)\le cc_D g(\rho(x,y))\le c^2c_D G(x,y).
\end{eqnarray*} 
\end{proof} 

\begin{remarks}{\rm
1. If $X$ is a topological space and each function $G(\cdot,x)$, $x\in X$,  
is lower semicontinuous and bounded at infinity,
the $\rho$ is a metric for the topology of $X$.

2. If (iii) holds, then, for every $\eta\in (0,1)$, there exists $\a\in (0,1)$ such that 
$g(r)\le \eta g(\a r)$. Indeed, it suffices to choose $k\in\nat$ such that $\eta_0^k<\eta$
and to take $\a:=\a^k$.
}
\end{remarks}

\bibliographystyle{plain} %  {alpha}
%\bibliography{lit_bank.bib}

\begin{thebibliography}{10}


\bibitem{aikawa-borichev}
H.~Aikawa and A.A.~Borichev.
\newblock Quasiadditivity and measure property of capacity and the tangential boundary behavior 
of harmonic functions.
\newblock{\em Trans. Amer. Math. Soc.}, 348:1013--1030, 1996.

\bibitem{BH}
J.~Bliedtner and W.~Hansen.
\newblock {\em {Potential Theory -- An Analytic and Probabilistic Approach to
  Balayage}}.
\newblock Universitext. Springer, Berlin, 1986.

\bibitem{gardiner-ghergu}
S.J. Gardiner and M.~Ghergu.
\newblock Champagne subregions of the unit ball with unavoidable bubbles.
\newblock {\em Ann. Acad. Sci. Fenn. Math.}, 35(1):321--329, 2010.

 \bibitem{grzywny}
 T.~Grzywny.
 \newblock On Harnack inequality and H{\"o}lder regularity for isotropic unimodal L{\'e}vy processes. 
 \newblock {\em Potential Anal.} 41: 1--29, 2014.

\bibitem{H-uniform}
W.~Hansen.
\newblock Uniform boundary {H}arnack principle and generalized triangle property.
\newblock{\em J. Funct. Anal.}, 226: 452--484, 2005.

\bibitem{H-course}
W. Hansen.
\newblock{\em Three views on potential theory}.
\newblock A course at Charles University (Prague), Spring 2008.
\newblock http://www.karlin.mff.cuni.cz/~hansen/lecture/ course-07012009.pdf. 

\bibitem{H-fuku}
W. Hansen
\newblock{ Unavoidable collections of balls for processes with isotropic unimodal Green function}.
In \newblock{\em Festschrift Masatoshi Fukushima} (eds. Z.-Q. Chen, N. Jacob, M. Takeda, T. Uemura), World Scientific Press, 2015.

\bibitem{convexity}
W.~Hansen and I.~Netuka.
\newblock Convexity properties of harmonic measures.
\newblock {\em Adv. Math.}, 218(4):1181--1223, 2008.

\bibitem{HN-champagne}
W. Hansen and I. Netuka.
\newblock Champagne subdomains with unavoidable bubbles.
\newblock{\em Adv. Math.} 244:106--116, 2013.

\bibitem{HN-rep-potential}
W. Hansen and I. Netuka.
\newblock Representation of potentials.
  \newblock{\em Rev. Roumaine Math. Pures Appl.}, 59: 93--104, 2014.

\bibitem{HN-unavoidable} 
W. Hansen and I. Netuka.
\newblock Unavoidable sets and harmonic measures living on small sets.
\newblock  {\em Proc. London Math. Soc.} 109: 1601--1629, 2014.

% \bibitem{HN-harnack}
% W. Hansen and I. Netuka.
% \newblock Harnack inequalities for Hunt processes with Green function.

\bibitem{heinonen}
J.~Heinonen.
\newblock {\em Lectures on analysis on metric spaces}.
\newblock Springer, New York, 2001.

\bibitem{maagli-87}
H. Maagli.
\newblock Repr\'esentation int\'egrale des potentiels.
\newblock In {\em S\'eminaire de {T}h\'eorie du {P}otentiel, {P}aris, {N}o.\
  8}, volume 1235 of {\em Lecture Notes in Math.}, pages 114--119. Springer,
  Berlin, 1987.



 \bibitem{mimica-vondracek}
 A.~Mimica and Z.~Vondra\v cek.
 \newblock Unavoidable collections of balls for isotropic L{\'e}vy processes.
  \newblock{\em Stochastic Process. Appl.}, 124(3):1303--1334, 2014.
% arXiv:1301.2441v2.

\end{thebibliography}
\def\cprime{$'$} \def\cprime{$'$}
%\end{document}

{\small \noindent 
Wolfhard Hansen,
Fakult\"at f\"ur Mathematik,
Universit\"at Bielefeld,
33501 Bielefeld, Germany, e-mail:
 hansen$@$math.uni-bielefeld.de}

\end{document}